\documentclass[11pt,a4paper,reqno]{amsart}
\usepackage{amssymb,amsmath,amsthm,amstext,amsfonts}
\usepackage[dvips]{graphicx}
\usepackage{psfrag}
\usepackage{color}
\usepackage[ansinew]{inputenc}
\usepackage{url}

\makeatletter \@addtoreset{equation}{section} \makeatother

\renewcommand\thetable{\thesection.\@arabic\c@table}

\theoremstyle{plain}
\newtheorem{maintheorem}{Theorem}

\newtheorem{maincorollary}{Corollary}

\newtheorem{theorem}{Theorem }[section]
\newtheorem{proposition}[theorem]{Proposition}
\newtheorem{lemma}[theorem]{Lemma}
\newtheorem{corollary}[theorem]{Corollary}

\theoremstyle{definition} \theoremstyle{remark}
\newtheorem{remark}[theorem]{Remark}
\newtheorem{example}[theorem]{Example}
\newtheorem{definition}[theorem]{Definition}

\newcommand{\al} {\alpha}       
\newcommand{\be} {\beta}        
\newcommand{\ga} {\gamma}    
\newcommand{\de} {\delta}       \newcommand{\De}{\Delta}

\newcommand{\vep}{\varepsilon}

\newcommand{\la} {\lambda}      \newcommand{\La}{\Lambda}

\newcommand{\si} {\sigma}

\newcommand{\om} {\omega}

\newcommand{\supp}{\operatorname{supp}}

\newcommand{\dist}{\operatorname{dist}}

\newcommand{\Leb}{\operatorname{Leb}}

\newcommand{\cH}{{\mathcal H}}
\newcommand{\cC}{\mathcal{C}}

\newcommand{\cR}{\mathcal{R}}
\newcommand{\cP}{\mathcal{P}}

\newcommand{\cE}{\mathcal{E}}

\newcommand{\cO}{\mathcal{O}}
\newcommand{\cG}{\mathcal{G}}

\newcommand{\cU}{\mathcal{U}}

\newcommand{\cA}{\mathcal{A}}

\begin{document}

\title{Statistical properties of generalized Viana maps}

\author{Paulo Varandas}
\address{Departamento de Matem\'atica, Universidade Federal da Bahia\\
  Av. Ademar de Barros s/n, 40170-110 Salvador, Brazil.}
\email{paulo.varandas@ufba.br \\ \url{http://www.pgmat.ufba.br/varandas/}}

\begin{abstract}
We study quadratic skew-products with parameters driven over piecewise expanding and Markov interval maps 
with countable many inverse branches, a generalization of the class of maps introduced by 
Viana \cite{Vi97}. In particular we construct a class of multidimensional non-uniformly expanding attractors that 
exhibit both  critical points and discontinuities and prove existence and uniqueness of an SRB measure with 
stretched-exponential decay of correlations, stretched-exponential large deviations and satisfying some limit laws.
Moreover, generically such maps admit the coexistence of a dense subset of points with negative
central Lyapunov exponent together with a full Lebesgue measure subset of points which have positive 
Lyapunov exponents in all directions. Finally, we discuss the existence of SRB measures for skew-products 
associated to hyperbolic parameters by the study of fibered hyperbolic maps.
\end{abstract}



\date{\today}

\maketitle 

\section{Introduction}

Since the 1960's, when the concept of uniform hyperbolicity was coined by Smale in~\cite{Sm67}, a
relevant question in dynamical systems is to construct examples that exhibit the hyperbolic features 
described by the theory. 
In fact, Hunt and Mackay~\cite{HM03} proved that uniformly hyperbolic dynamical systems, among 
which Smale's horseshoe is a paradigmatic example, arise naturally in physical systems. 
On other direction, simple one-dimensional examples arising from populational dynamics led to consider the quadratic family $T_a(x)=a x(1-x)$, or equivalently $f_a(x)=1-ax^2$, that 
despite the simple formulation presents very rich and complex dynamics. 
In fact, it follows from pioneering works by Jakobson, Benedicks and Carleson~\cite{Jak81, BC85} that there exists 
a positive Lebesgue measure set 
$\Delta\subset (0,2]$ such that $f_a(x)=1-ax^2$ has an absolutely continuous ergodic probability measure 
with positive Lyapunov exponent for all $a\in \De$, ie,  is non-uniformly hyperbolic.
Later, Graczyk and \'Swi\c atek~\cite{GS97} and Lyubich~\cite{Lyu97} proved that $f_a$ is hyperbolic for an open and dense set of parameters $a\in(0,2]$. 

Much more recently, in a major breakthrough, in~\cite{KSvS07a}  Kozlovski, Shen and van Strien proved 
that hyperbolicity is open and dense among $C^k$ maps of the interval or the circle $(k\ge 1)$. In particular this shows that although persistent, meaning a positive Lebesgue measure phenomenon for parametrized families,  
one-dimensional examples with robust  non-uniform hyperbolicity do not exist. 

 In a higher dimensional setting the existence of robust non-uniform hyperbolicity was addressed by 
Viana~\cite{Vi97} that introduced a class of $C^3$-maps of the cyclinder:
they are any small $C^3$-perturbations $\varphi$ of the skew-product transformations
\begin{equation*}\label{eq:Viana.maps}
\begin{array}{cccc}
\varphi_\al: & S^1 \times I & \to & S^1 \times I \\
		 & (\theta,x)& \mapsto & (g(\theta), f_\al(\theta,x))
\end{array}
\end{equation*}
where $g$ is an expanding map of the unit circle $S^1$ with  $|g'(\theta)|\ge d$ where $d\ge 16$ and 
$f_\al(\theta,x)=a_0+\al\sin(2\pi\theta) -x^2$ for some small $\al>0$ and parameter $a_0\in (0,2]$ such that
the quadratic map $h(x)=a_0-x^2$ is of Misiurewicz type. Despite the presence of a critical region, Viana 
proved that this class of transformations of  the cylinder $S^1\times \mathbb R$ have positive Lyapunov exponents 
in every direction, that is
$$
\liminf_{n\to\infty} \frac1n \log \|D\varphi^n(\theta,x) v \|>0
$$  
for Lebesgue almost every $(\theta,x)$ and all $v\in T_{(\theta,x)} (S^1\times \mathbb R)$.
Building over this,  Alves~\cite{Al00} proved that there is a unique $\varphi$-invariant probability 
measure absolutely continuous with respect to Lebesgue. In other words, this class of $C^3$-endomorphisms of 
the cylinder present  a robust non-uniform hyperbolicity phenomenon.  
More recent contributions and extensions include the ones by Gouzel \cite{Gou07} on the skew products
with curve of neutral fixed points, by Buzzi, Sester and Tsujii~\cite{BST03} 
for $C^\infty$-perturbations of the skew product $\varphi_\al$ with a weaker condition $d\ge 2$ and later on by Schnellmann~\cite{Sc08} that considered 
$\beta$-transformations and by Schnellmann, Gao, Shen~\cite{Sc09,GS12} considering Misiurewicz-Thurston 
quadratic maps as the base dynamics. In these results the proof that Lebesgue almost every point has only 
positive Lyapunov exponents exploits, in the spirit of \cite{Vi97}, some weak hyperbolicity condition, namely dominated decomposition.

An important challenge in dynamics is to construct multidimensional attractors with critical behaviour 
without dominated splittings but persistence of positive Lyapunov exponents in parameter space.
It was proposed by Bonatti, D\'iaz and Viana in ~\cite{BDV05} 
that such phenomena might occur in a parametrized family $F(x,y) = (a(x,y)-x^2, b(x,y)-y^2))$. 
As mentioned above, important contributions were given in \cite{Sc09,GS12} where the authors proved non-uniform expansion for skew-product  of quadratic maps over a Misiurewicz-Thurston quadratic map. 
Since parameters corresponding to Misiurewicz-Thurston quadratic maps have zero Lebesgue measure in the 
parameter space then the previous question remains open.

Our purpose in this paper is to study a class of quadratic skew-products over a Markov expanding map of 
the interval with at most  countably many inverse branches: skew-products  $\varphi_\al(\theta,x)= 
(g(\theta), f_\al(\theta,x))$ with $g$ piecewise expanding Markov map of the unit interval and  $f_\al(\theta,x)=a_0+\al\sin(2\pi\theta) -x^2$ for some small $\al>0$ and parameter $a_0\in (0,2]$. 
One important motivation to consider skew products of quadratic maps over 
expanding dynamics with infinitely many branches is to understand if the technique of  inducing can be an useful 
approach to the previous question since non-uniform expansion is well known to be related with inducing schemes leading to piecewise expanding maps with infinitely many branches.
We prove that the behaviour of this class of transformations has different behaviours depending on the parameter $a_0$.

In the one hand, if $h(x)=a_0-x^2$ is a Misiurewicz quadratic map then
we overcome the difficulty caused by the presence of critical points and infinitely many invertibility branches 
to prove the existence of two positive Lyapunov exponents at Lebesgue almost every point, that there
exists a unique absolutely continuous invariant measure and that it satisfies good statistical properties.
Moreover, 
we also prove that generically such transformations admit the coexistence of the full Lebesgue measure set of points which have only positive Lyapunov exponents together with a dense set of points with one positive and one negative Lyapunov exponents, a fact that was unkown even in the context of Viana maps in \cite{Vi97}. 
On the other hand, if the quadratic map $h(x)=a_0-x^2$ is hyperbolic then the skew product $\varphi_\al$
admits an hyperbolic SRB measure, provided that $\al$ is small. In fact this will be consequence of a more
result for fibered hyperbolic polynomials.

The difficulties in proving our results are substantially different depending if one is considers perturbations
of a Misiurewicz or a hyperbolic  quadratic map. In the first case, the strategy of  \cite{Vi97} for the positivity of Lyapunov exponents Lebesgue almost everywhere is to consider the iteration of admissible curves (a class of curves 
preserved under iteration) and to prove that Lebesgue almost every point in each admissible curve has a slow 
recurrence condition to the critical region and thus has exponential growth of the derivative. The key 
combinatorial argument in \cite{Vi97} does not hold in our setting since any admissible curve is mapped onto
countably many of them. 
In our Proposition~\ref{prop:main1} we overcome this difficulty and obtain a statistical argument showing
that, up to consider a finite iteration $\varphi^N$ of the skew-product $\varphi$, the proportion of points
in any admissible curve whose iterate intersects a small neighborhood of the critical region is bounded by
the Lebesgue measure of the corresponding neighborhood.
Then, a large deviations argument shows that the measure of points in admissible curves that exhibit fast 
recurrence to a neighborhood of the critical region decrease subexponentially and, consequently, Lebesgue
almost every point has two positive Lyapunov exponents.
In the second case, corresponding to perturbation of hyperbolic quadratic maps, the problem can be understood
as a random composition of nearby quadratic maps and so we consider a more general setting of fibered 
nearby hyperbolic interval maps. 
The hyperbolicity of the Julia set for fibered polynomial maps in the complex variable setting has been studied by Jonsson~\cite{Jo97} and Sester~\cite{Se99}.  
Here we use a small neighborhood of the periodic atracting orbit to prove that there exists a trapping region 
and an attractor. The transversality of admissible curves implies that the attractor does not coincide with the 
continuation of the periodic attracting orbits, but using a uniforme contraction property one proves that it still 
supports a unique hyperbolic SRB measure. Roughly, the complement of the basin of attraction is formed by 
points whose random iteration of quadratic maps that does not intersect the trapping region. 
This paper is organized as follows. In Section~\ref{s.statements} we recall some definitions and state our main
results. Some preliminary results on one-dimensional dynamics and admissible curves are given along Section~\ref{sec:preliminaries}. Finally, the proofs of the main results are given on Sections~\ref{sec:positive.exp},~\ref{sec:proofs} and~\ref{sec:periodic}. 

\section{Statement of Main Results}\label{s.statements}

In this section we present the necessary  definitions to state our main results. 

\subsection{Setting}

Let $\cP=\{\om_i\}_{i\in S, S\subset \mathbb N_0}$ be an at most countable partition of the unit interval $(0,1]$ by subintervals and $g:(0,1]\to (0,1]$ be a $C^{3}$ piecewise differentiable map. We will say that $g$ is a \emph{Markov expanding} map if $g(\om_i)=(0,1]$ and the restriction $g\mid_{\om_i}$ is a 
$C^{3}$ diffeomorphism with a $C^3$ extension to the closure for any $i\in S$,  $|g'(\theta)|\ge d\ge 16$
for all $\theta\in (0,1]$, and there exists $K>0$ so that $|g''|\leq K |g'|^2$. The later is the so called R\`enyi condition,
which is a sufficient condition to obtain the bounded distortion property in Subsection~\ref{subsec:hyperbolicity}.
Throughout we assume that $\log |g'| \in L^1(\Leb)$. This is a natural assumption to obtain finite positive Lyapunov
exponent for $g$ and related with the size of smaller intervals of $\mathcal P$.
Now we introduce the family of skew-products of the space $(0,1]\times \mathbb R$ with 
countably many inverse branches.

\begin{definition}\label{def:generalized}
We say that a piecewise $C^3$ map $\varphi\colon (0,1]\times \mathbb R \to (0,1]\times \mathbb R$ is a \emph{generalized Viana map} if it is a skew-product given by 
\begin{equation*}
\begin{array}{cccc}
\varphi_\al: & (0,1]\times \mathbb R & \to & (0,1]\times \mathbb R \\
		& (\theta,x)& \mapsto & (g(\theta), f_\al(\theta,x))
\end{array}
\end{equation*}
where $g$ is a piecewise linear Markov expanding map on $(0,1]$ and  $f_\al(\theta,x)=a_0+\al\sin(2\pi\theta) -x^2$ for some 
$\al>0$ and parameter $a_0\in (0,2]$. 
\end{definition}

The assumptions on the parameter $a_0$ will be crucial. Recall that the quadratic map $h(x)=a_0-x^2$ 
is Misiurewicz provided that the critical point  is pre-periodic repelling.
It is not hard to check that there 
exists an interval $I_0\subset \mathbb [h^2(0), h(0)]$ such that $\varphi_\al((0,1]\times I_0) \subset (0,1]\times I_0$ 
for every $\al>0$ small enough. Then we define 
the attractor $\La=\La(\varphi_\al)$ for $\varphi_\al$ by
$$
\La(\varphi_\al) = \bigcap_{n\ge 0} \varphi_\al^n \left(  (0,1]\times I_0 \right)
$$
and consider the restriction $\varphi_\al |_{\La}$.
Let us mention that although it seems reasonable that some other classes of  infinitely branched 
interval expanding maps can be considered as base dynamics without the  R\`enyi assumption 
some condition on the decay of the size of the partition elements should be necessary (e.g. otherwise
could exist SRB measures without finite positive Lyapunov exponent). 

Observe that we defined generalized Viana maps as a class of skew-products. We shall consider in this space
the topology that we now describe. Assume, without loss of generality, that $S=\mathbb N_0$. 
Given $\vep>0$ we say that $\varphi: (0,1]\times \mathbb R \to (0,1]\times \mathbb R$ is 
\emph{$\vep$-$C^3$-close }to the skew-product $\varphi_\al$ above if 
$\varphi(\theta,x)=(\tilde g(\theta), \tilde f(\theta,x))$ is a piecewise $C^3$ map and satisfies:
\begin{itemize}
\item[(i)] $\varphi((0,1]\times I_0) \subset (0,1]\times I_0$, the map
	  $\tilde g$ is a Markov expanding map on $(0,1]$;
\item[(ii)] $(0,1]=\bigcup_{i\in S} \tilde \omega_i$ and $\tilde g_i\mid_{\tilde\omega_i}:\tilde\omega_i\to (0,1]$
	is a diffeomorphism that admits a $C^3$ extension to the boundary and 
	 is a Markov partition for $\tilde g$ 
\item[(iii)] the renormalized maps $R_i g: \om_i \times \mathbb R\to (0,1]$ 
	 and $R_i f: \om_i\times \mathbb R \to  \mathbb R$ given respectively by
	$$
	R_i g(\theta)= \tilde g \left(
			\tilde \theta_{i+1} + \frac{|\tilde\omega_i|}{|\tilde\omega_i|} (\theta -\theta_{i+1})
			\right)
		\quad\text{and}\quad
	R_i f(\theta,x)= \tilde f \left(
			\tilde \theta_{i+1} + \frac{|\tilde\omega_i|}{|\tilde\omega_i|} (\theta -\theta_{i+1}),x
			\right)
	$$
	satisfy $\sup_{x} \|g(\cdot)-R_i g(\cdot,x)\|_{C^3}<\vep$ and $\|f|_{\om_i\times \mathbb R}-R_i f\|_{C^3}<\vep$,
	where $|\cdot|$ denotes the Lebesgue measure of the interval $\omega_i$ with boundary points $\theta_i$ and 	$\theta_{i+1}$.
\end{itemize}

Let us make some comments on our assumptions. We will assume for notational simplicity that the partition $\cP$
is preserved under perturbations, in which case $C^3$ perturbation coincides with the usual notion for interval maps.
Condition (i) implies that the perturbed map
$\varphi : (0,1]\times \mathbb R\to (0,1]\times \mathbb R$ is a skew-product with countably many domains of invertibility over a Markov expanding map.
These are natural assumptions if one assumes the base dynamics to be induced map from some 
one-dimensional nonuniformly expanding map.
Condition (ii) implies the domains of invertibility of $\varphi$ to be close to those of 
$\varphi_\al$ and that the dynamics in each domain is $C^3$-close to the original one. So, these assumptions
require the map $\varphi$ to be close to $\varphi_\al$ from both the topological and the differentiable viewpoints. 
Clearly, this class of maps include the skew-products considered in \cite{Vi97}.
\subsection{Statement of results}\label{s.thermodynamics}

We are now in a position to state our main results. 

\begin{maintheorem}\label{thm:Main.Estimate}
Consider the skew-product $\varphi_\al: (0,1] \times I  \to  (0,1] \times I$ given by
$\varphi_\al(\theta,x)= (g(\theta), f_\al(\theta,x))$ and such that $h(x)=a_0-x^2$ is Misiurewicz.
Then there exists $c>0$ such that for any small $\al$ it holds
$$
\liminf_{n\to\infty} \frac1n \log \| D\varphi_\al^n(\theta,x) v\| \geq c >0
$$
for Lebesgue almost every $(\theta,x)$ and every $v\in\mathbb R^2\setminus\{0\}$. Moreover, there exists $\vep>0$ 
such that the same property holds for every $\varphi$ that is $\vep$-$C^3$-close to $\varphi_\al$.
\end{maintheorem}

As a byproduct of the proof we obtain some estimates on the decay of the first time at which some hyperbolicity is
obtained. These are known as hyperbolic times (see \cite{Al00} for some details). In consequence, one can 
use the works of Ara\'ujo, Solano~\cite{ArS11} or Pinheiro~\cite{Pi11} to build an inducing scheme 
and deduce the existence of an SRB measure with good statistical properties. 
Recall that a $\varphi$-invariant and ergodic probability measure $\mu$ is an \emph{SRB measure} if its 
basin of attraction 
$$
B(\mu)
=\Big\{(\theta,x)\in (0,1]\times I_0 : \frac1n\sum_{j=0}^{n-1} \de_{\varphi^j(\theta,x)}  \xrightarrow[]{w^*} \mu \Big\}
$$
has positive Lebesgue measure. Here we establish not only uniqueness of the SRB measure as we obtain 
several important statistical properties. Let $\mathcal H_\beta$ be denote the space of $\beta$-H\"older continuous observables.
We obtain the following:

\begin{maintheorem}\label{thm:SRBmeasure}
Let $\varphi_\al: (0,1] \times I  \to  (0,1] \times I$ be a generalized Viana map 
$\varphi_\al(\theta,x)= (g(\theta), f_\al(\theta,x))$ such that $h(x)=a_0-x^2$ is Misiurewicz.
Then, for any small $\al>0$:
\begin{enumerate}
\item  $\varphi_\al$ is topologically mixing;
\item There exists a unique $\varphi_\al$-invariant measure $\mu_\al$ that is absolutely continuous 
	with respect to Lebesgue on the attractor $\Lambda(\varphi_\al)$; 
\item $\mu_\al$ has stretched-exponential decay of correlations, that is, there exists $C>0$ and $\tau\in(0,1)$
	such that
	$$
	\left|\int (h_1 \circ \varphi_\al^n) h_2 \, d\mu_\al- \int h_1 \, d\mu_\al . \int h_2 \, d\mu_\al \right|
		\leq C e^{-\tau \sqrt{n}}\|h_1\|_\infty \|h_2\|_\beta
	$$
	for all large $n$ and observables $h_1\in L^\infty(\mu)$ and $h_2\in \cH_\beta$;
\item $\mu_\al$  has stretched-exponential large deviations, meaning that there exists $\zeta\in (0,\frac12)$ such that 
	for all $\de>0$ and $h \in \cH_\beta$ there exists $\gamma>0$ satisfying
	$$
	\mu\Big( \Big|\frac1n \sum_{j=0}^{n-1} h\circ \varphi^j - \int h \, d\mu_\al \Big|>\de \Big)
		\leq e^{-\gamma n^{\zeta}}
		\text{  for all large $n$;}
	$$
\item $\mu_\al$ satisfies the central limit theorem, the almost sure invariance principle, 
 	the local limit theorem and the Berry-Esseen theorem for H\"older observables 
\end{enumerate}
Furthermore, all these properties hold for every $\varphi$ that is $C^3$-close enough to $\varphi_\al$.
\end{maintheorem}

Our strategy to deduce the later ergodic properties is to use recent contributions to the study of 
stretched-exponential large deviations and limit theorems using Markov induced maps e.g. by Melbourne and 
Nicol~\cite{MN08}, Alves, Luzzatto, Freitas, Vaienti \cite{ALFV11} or Alves and Schnelmann \cite{AS13}. 
Our next main result concerns the coexistence of a dense set of points with a negative and a positive
Lyapunov exponents together with a full Lebesgue measure set of points with only positive Lyapunov exponents.  

\begin{maintheorem}\label{thm:generic.skew}
Let $\varphi_\al: (0,1] \times I  \to  (0,1] \times I$ be a generalized Viana map 
$\varphi_\al(\theta,x)= (g(\theta), f_\al(\theta,x))$ such that $h(x)=a_0-x^2$ is Misiurewicz and
let $\mathcal V$ be a $C^3$-open set of generalized Viana maps. Then there exists an open and dense 
subset $\cA\subset \mathcal V$ such that for every $\varphi \in\cA$
\begin{enumerate}
\item there is a dense set of points $D\subset \La(\varphi)$ with a negative Lyapunov exponent,
	that is,
	$$
	\limsup_{n\to\infty} \frac1n \log \Big\|D\varphi^n(\theta,x) \frac{\partial}{\partial x}\Big\|<0
		\quad \text{for all } (\theta,x) \in D;
	$$ 
\item Lebesgue almost every point in $\La(\varphi)$ has two positive Lyapunov exponents.
\end{enumerate}
\end{maintheorem}

In view of the previous theorem an interesting question is to understand if, at least generically, all 
Lyapunov exponents are bounded away from zero.
The later goes in the direction of understanding possible phase transitions for the topological
pressure $P(t)$ of the generalized Viana-map $\varphi_\al$ with respect to the family of potentials  
$\psi_{\al,t}= -t \log |\partial_x \varphi_\al|$, that is, parameters $t$ 
such that  $\psi_{\al,t}$ has none or more than one equilibrium state.
For that purpose it would be important to characterize the range of the entropy 
among invariant and ergodic measures with one negative Lyapunov exponent. 

This final part of the section is devoted to a better understanding of the dynamics of quadratic skew-products  
where the parameters are driven among hyperbolic parameters. In general, to obtain hiperbolicity for the 
composition of fibered polynomials is a hard question. Jonsson~\cite{Jo97} and Sester~\cite{Se99} studied this
topic in the complex variable setting and, in particular,  obtained conditions equivalent to hyperbolicity.  With this in 
mind we state our last main result. For simplicity, we will say that an ergodic measure is \emph{hyperbolic} if is has only nonzero Lyapunov exponents and not all of the same sign.

\begin{maintheorem}\label{thm:hiperbolicidade.fibrada}
Let $I$ denote a closed interval. Assume that $X$ is a compact Riemannian manifold, that a continuous
map $S:X \to X$ admits a unique SRB measure $\mu_S$, that $T: I \to I$ is a hyperbolic polynomial with 
negative Schwarzian derivative and $a: X \to \mathbb R$ is  a $C^3$-smooth function. There exists
$\vep>0$ such that if $\|a\|_{C^3}<\vep$ then the skew product 
$$
\psi : (x,y)\mapsto (S(x), T(y)+a(x))
$$
is such that $\psi$ has an ergodic SRB measure $\nu$ whose basin of attraction contains Lebesgue almost every
 point in some open and proper subset of $X\times I$. If, in addition, $\mu_S$ is hyperbolic then $\nu$ is also hyperbolic.
\end{maintheorem}

Let us mention that it is not clear to us wether the transversality condition of admissible curves can be
used to prove that the hyperbolic SRB measure is absolutely continuous with respect to Lebesgue. 
Related results were obtained by Tsujii~\cite{Tsu01} and Volk~\cite{Vo11}.
Finally, the previous result applies for different coupling functions $S$ as rotations, $C^2$-Maneville-Pommeau maps and, in particular, directly for Viana maps whose parameters are driven along hyperbolic one as follows. 

\begin{maincorollary}\label{cor:hyperbolicity.Viana}
Consider the skew-product $\varphi_\al: S^1 \times I  \to  S^1 \times I$ given by
$\varphi_\al(\theta,x)= (g(\theta), f_\al(\theta,x))$, where $g(\theta)=d\theta (\!\!\mod 1)$ and
$f_\al(\theta,x)=a_0+\al \, \sin(2\pi\theta)-x^1$ for $a_0$ so that  $h(x)=a_0-x^2$ is hyperbolic.
Then for every small $\al>0$ the map $\varphi_{\al}$ has an SRB  measure $\nu$, it is hyperbolic 
and supported in a proper attractor.
Moreover, the same property holds for every $\varphi$ that is $C^3$-close enough to $\varphi_\al$.
\end{maincorollary}

Let us mention that the SRB measure above is supported on a  topological atractor $\cG\subset (0,1]\times I_0$
as described in detail in Section~\ref{sec:SRBs}. Finally, some interesting questions are to understand 
if the complement $K$ of the attractor $\mathcal G$ is such that the intersection $K_\theta$ with the fiber over 
$\theta$ is a totally disconnected, compact set of zero Lebesgue measure. For that it would be necessary to
prove that $K$ is an expanding set for $\varphi_\al$.
Moreover, since almost every parameter is regular or stochastic for the quadratic map 
it would be interesting to understand wether either of Theorem~A or Corollary~\ref{cor:hyperbolicity.Viana} 
hold for quadratic skew-products and Lebesgue almost every parameter $a\in(0,2]$.

\subsection{Some applications}
Let us finish this section with some examples. 

\begin{example}[Viana maps]
The class of maps considered in~\cite{Vi97} fit in the previous setting. In fact, assume $d\geq 16$ and take
the  Markov expanding map on $(0,1]$ given by $g(\theta)=d\theta-[d\theta]$ (where $[\cdot]$ stands for the integer part). Then $g$ admits a $C^3$-extension to the boundary elements and one can identify the boundary points of $(0,1]$ to obtain the $C^3$ expanding map on the circle ${\mathbb S}^1$ given by $g(\theta)=d\theta (\!\!\!\mod 1)$,
thus recovering the previous setting. 
In particular, in this context Theorems~\ref{thm:Main.Estimate} and~\ref{thm:SRBmeasure} 
are consequences of \cite{Vi97,Al00}. 

In this context, it follows from Theorems~\ref{thm:generic.skew}  that $C^1$-generic transformations in the $C^3$
neighborhood of Viana maps exhibit coexistence of a dense set of points with one negative Lyapunov exponent  
while Lebesgue almost every point has only positive Lyapunov exponents. 
Finally, it follows from Theorem~\ref{cor:hyperbolicity.Viana} that quadratic skew-products with parameters 
driven among hyperbolic ones admit a unique SRB measure, it is hyperbolic and the complement of its basin
of attraction is an expanding Cantor set of lines.
\end{example}

In the next class of examples we present a robust class of Markov expanding maps with discontinuities and 
infinitely many invertibility domains.

\begin{example}[Quadratic skew-products over piecewise linear expanding maps]
Let $\cP$ be an arbitrary countable partition of the unit interval $(0,1]$ in subintervals $(\om_i)_{i\in S}$
with size smaller or equal to $\frac1d$ and let $g_0$ be piecewise linear satisfying $g(\om_i)=(0,1]$. Since 
$|g_0'(\theta)|\ge d$ and $g_0'(\theta)'=0$ for all $\theta\in(0,1]$ then it is clear that $g_0$ is a piecewise
Markov expanding map and satisfies the R\`enyi condition. See Figure~1.
\begin{figure}[htb]
\begin{center}
\includegraphics[height=1.6in]{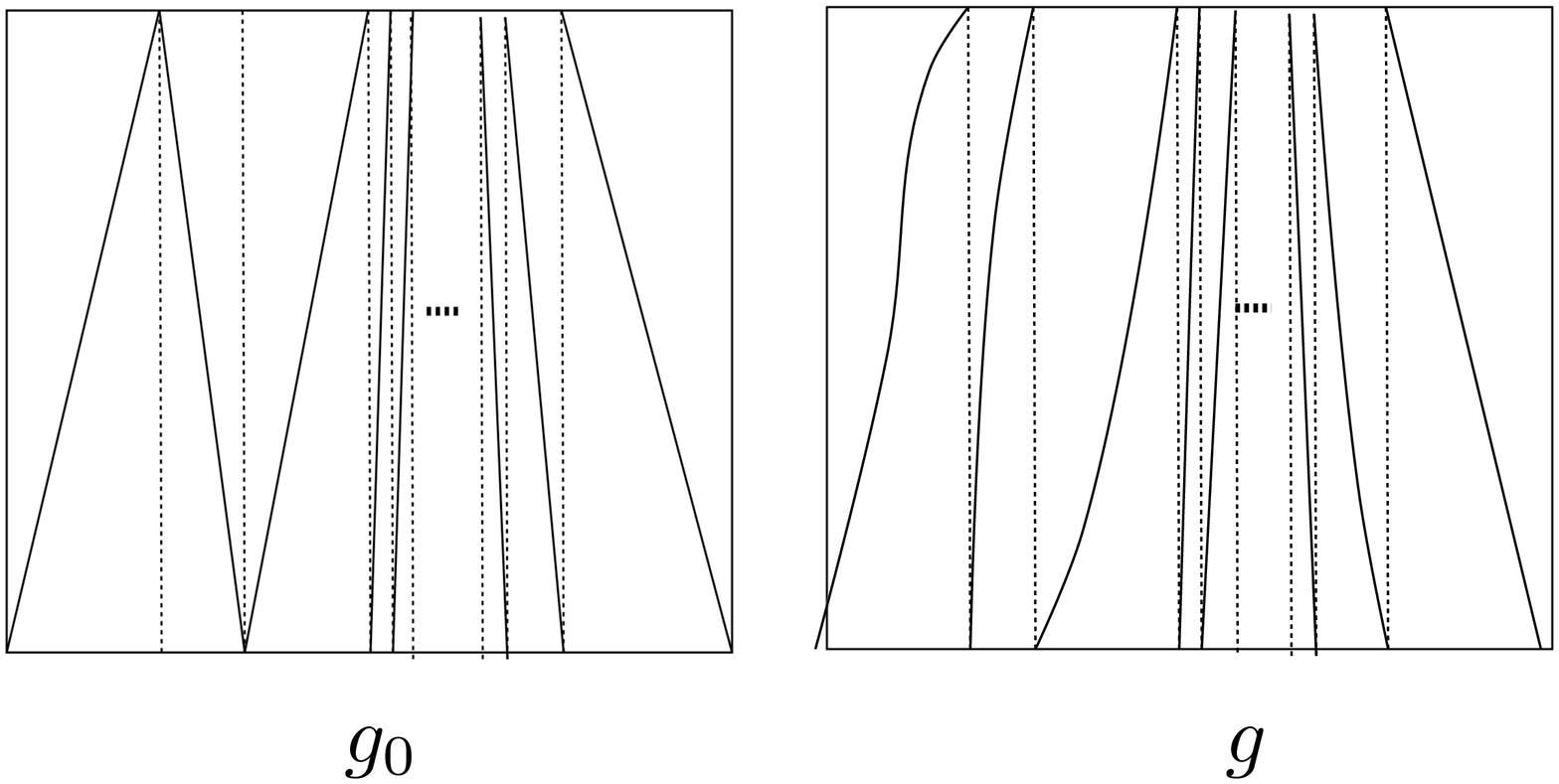}
\end{center}
\vspace{-.6cm}
\caption{Small perturbation of piecewise linear expanding map}
\end{figure}

Moreover, if $\tilde g$ is a Markov expanding map that is $\vep$-$C^3$-close enough to $g$ then 
it follows from Lemma~\ref{lem:open} that it also satisfies the R\`enyi  condition thus satisfying all 
the hypothesis to be used as base dynamics.
\end{example}

\section{Preliminaries}\label{sec:preliminaries}

In this section we recall some definitions and preliminaries that will be used in the proof of the main results.

\subsection{One dimensional dynamics}\label{subsec:hyperbolicity}

\subsubsection*{Combinatorial description of Markov expanding maps}\label{subsec:onedim}

Here we describe the Markov expanding maps $g$ 
from the combinatorial point of view. Let $\cP=\{\om_i\}_{i\in S, S\subset {\mathbb N}}$ be the
Markov partition for $g$.
Then there is a semi-conjugacy between the dynamics of $g$ 
and the full shift $\sigma : S^{\mathbb N} \to S^{\mathbb N}$ given by
$$
\sigma (s_0, s_1, s_2, \dots) = ( s_1, s_2, s_3, \dots),
$$
where the semi-conjugation is given by the \emph{itinerary map} $\iota: {\mathbb N}^{\mathbb N} \to (0,1]$
defined as  $\iota(s_0, s_1, s_2, \dots)= \theta$ and $\theta$ is the only point in $(0,1]$ satisfying 
$g^j(\theta) \in \om_{s_j}$ for all $j$. 
Set $\cP^{(n)}=\bigvee_{j=0}^{n-1}g^{-j}(\cP)$ and for any partition element $\om\in\cP^{(n)}$  define 
$\iota_n=\iota_n(\om)=( s_1, s_2, \dots, s_{n-1})$ as its \emph{$n$-th itinerary}.
For simplicity we will denote by $\om_{(s_0, s_1, s_2, \dots, s_{n-1})}$ the element of $\cP^{(n)}$
whose itinerary is $(s_0, s_1, s_2, \dots, s_{n-1})$. This will be helpful to give a precise description of
points that visit to definite regions of the phase space.

\subsubsection*{R\`enyi condition} 

Here we show that  R\'enyi condition is an open property among Markov 
expanding maps and relate this with the bounded distortion property. 

\begin{lemma}\label{lem:open}
Assume that $g:(0,1] \to (0,1]$ is a $C^3$-Markov expanding map satisfying the R\`enyi condition 
$|g''|\le K |g'|^2$.  If $\|g-\tilde g\|_{C^3}<\vep$ for small $\vep$ then $\tilde g$ satisfies $|\tilde g ''| 
\le \tilde K |\tilde g'|^2 $ with $\tilde K = (d-\vep)^{-2} \vep+ (1-\vep)^{-2} K$.
\end{lemma} 

\begin{proof}
Assume that $\|g-\tilde g\|_{C^3}<\vep$. Using that $g$ is expanding it follows that $|g'-\tilde g'|<\vep<\vep |g'|$ 
and consequently
\begin{align*}
\frac{|\tilde g ''|}{|\tilde g'|^2} 
	& \leq \frac{|g ''| + |g ''-\tilde g''|}{|\tilde g'|^2} 
	\leq \frac{1}{(d-\vep)^2} \|g-\tilde g\|_{C^2} + \frac{|g ''|}{(|g'|-|g'-\tilde g'|)^2} \\
	& \leq \frac{1}{(d-\vep)^2} \|g-\tilde g\|_{C^2} + \frac{1}{(1-\vep)^2} \frac{|g''|}{ |g'|^2} 
	\leq \frac{1}{(d-\vep)^2} \|g-\tilde g\|_{C^2} + \frac{K}{(1-\vep)^2}.
\end{align*}
This proves that $\tilde g$ also satisfies the R\`enyi condition and proves the lemma. 
\end{proof}

In the second lemma we collect some bounded distortion estimates. 

\begin{lemma}\label{lem:bdistortion}
Let $g$ be a $C^3$-Markov expanding map satisfying $|g'(\theta)|\ge d$ and the R\`enyi condition 
$|g''|\le K |g'|^2$. Then, for all $\om\in \cP^{(n)}$ and $\theta_1,\theta_2\in \om$
$$
\exp\Big( -\frac{d K}{d-1} \Big)  
	\leq \frac{|(g^n)'(\theta_1)|}{|(g^n)'(\theta_2)|}
	\leq \exp\Big( \frac{d K}{d-1} \Big).
$$
\end{lemma}

\begin{proof}
Let $\theta_1,\theta_2\in \om$ for some $\om\in \cP^{(n)}$ be given. We may assume, without loss of generality,
that $g^n\mid_\om$ is increasing. Then
\begin{align*}
\Big| \log \frac{(g^n)'(\theta_1)}{(g^n)'(\theta_2)} \Big|
	& \leq \sum_{j=0}^{n-1}\Big|  \log  (g'(g^j(\theta_1))) - \log (g'(g^j(\theta_2)))  \Big| \\
	& = \sum_{j=0}^{n-1}\Big|  \int_{g^j(\theta_1)}^{g^j(\theta_2)}  \frac{g''(\theta)}{g'(\theta)} \; d\theta\Big| 
	 \leq K \sum_{j=0}^{n-1}\Big|  \int_{g^j(\theta_1)}^{g^j(\theta_2)} g'(\theta) \; d\theta\Big| \\
	 & = K \sum_{j=1}^{n} | g^j(\theta_1) - g^j(\theta_2) |
	 \leq K \sum_{j=1}^{n} d^{-(n-j)} | g^n(\theta_1) - g^n(\theta_2) | 
\end{align*}
which is clearly bounded from above by $K( 1-d^{-1})^{-1}$. Since $\theta_1,\theta_2$ were arbitrary then the 
lower bound also holds. This finishes the proof of the lemma.
\end{proof}
 
Finally let us recall that it is well known that if $g$ is a Markov expanding map with the R\`enyi condition then 
there exists a unique absolutely continuous invariant probability measure. We will use the following strong Gibbs property for Lebesgue.

\begin{corollary}\label{cor:cylinders}
Let $g$ be a $C^3$-Markov expanding map satisfying $|g'(\theta)|\ge d$ and the R\`enyi condition 
$|g''|\le K |g'|^2$. Then, for all $\om\in\cP^{(n)}$ and $\theta^*\in \om$
$$
\exp\Big( - \frac{d K}{d-1} \Big)
	\leq \frac{\Leb(\om)}{  |(g^n)'(\theta^*)|^{-1}} 
	\leq \exp\Big( \frac{d K}{d-1} \Big)
$$
\end{corollary}

\begin{proof}
The proof is a simple application of the usual change of coordinates to the diffeomorphism 
$g^n\mid_\om: \om \to (0,1]$ by means that
$$
 \int_\om |(g^n)'(\theta)| \; d\theta
 	= \Leb(g^n(\om))
	=  \Leb(f(0,1]) = 1
$$
together with the previous bounded distortion estimates.
\end{proof}

\subsubsection*{Hyperbolicity in dimension one}

In \cite{Sm00}, Smale proposed the density of hyperbolicity in dimension one as one of the problems for the 21st
century.   Recall that a $C^1$ endomorphism  of a compact interval is \emph{hyperbolic} if it has finitely many 
hyperbolic attracting  periodic points and the complement of the basins of attraction is a hyperbolic set.    
Such major achievement was obtained by Kozlovski, Shen and van Strien.

\begin{theorem}(Theorem~2 in~\cite{KSvS07a})\label{thm:KSV}
Hyperbolic maps are dense in the space of $C^k$ maps of the compact interval or the circle
for $k=1,2, \dots, \infty, \omega$.
\end{theorem} 

\subsubsection*{Fibered expansion}\label{subsec:fe}

Now we collect some results on a mechanism to obtain expansion in the fiber direction. Roughly, 
expansion for random composition of perturbations of a Misiurewicz quadratic map $h(x)=a_0-x^2$ is obtained 
if orbits avoid the critical region and the loss of expansion in each return to the critical region is 
proportional to the return depth.  
More precisely, 

\begin{proposition} \label{prop:uniform.expansion}\cite[Lemmas~2.4 and~2.5]{Vi97}
Given $\al>0$, take $h(x)=a_0-x^2$, $f(\theta,x)=a_0+\al\sin(2\pi\theta)-x^2$
and $\varphi(\theta,x)=(g(\theta), f_\al(\theta,x))$. There are constants $0<\kappa<1$ and $0<\eta\leq \frac13 $
(depending only on $h$) and $\de_1, C_2>0$,  $\si_1,\si_2>1$ such that,
for every small $\al>0$ there exists $N(\al)\geq 1$ satisfying:
\begin{enumerate}
\item $K_0 \log\frac1\al \leq N(\al) \leq K_1 \log\frac1\al $ for some uniform constants $K_0, K_1>0$;
\item Given an interval $I\subset I_0$, for every $(\theta,x)\in (0,1]\times I$ with $|x|<2\sqrt{\al}$ the iterates $(\theta_j,x_j)=\varphi^j(\theta,x)$ 
	satisfy $|x_j|\geq \sqrt{\al}$ for every $j=1\dots N(\al)$;
\item $
	\prod_{j=0}^{N(\al)-1} |\partial_x f (\theta_j,x_j)| 
		\geq |x| \al^{-1+\eta}
	$
	for all $(\theta,x)\in (0,1]\times I$ with $|x|<2\sqrt{\al}$;
\item For every $(\theta,x)\in (0,1]\times I$ with $\sqrt{\al}\leq |x|<\de_1$ there exists $p(x)\leq N(\al)$ so that
	$
	\prod_{j=0}^{p(x)-1} |\partial_x f (\theta_j,x_j)| 
		\geq \frac1\kappa \si_1^{p(x)};
	$
\item $
	\prod_{j=0}^{n-1} |\partial_x f (\theta_j,x_j)| 
		\geq C_2 \sqrt{\al} \si_2^n
	$
	for every $(\theta,x)\in (0,1]\times I$ with $|x_j|\geq \sqrt{\al}$ for every $j=1,\dots, n-1$; and
\item $
	\prod_{j=0}^{n-1} |\partial_x f (\theta_j,x_j)| 
		\geq C_2 \si_2^n
	$
	for all $(\theta,x)\in (0,1]\times I$ such that $|x_j|\geq \sqrt{\al}$ with $j=1,\dots, n-1$ 
	and	$|x_n|\le \de_1$.
\end{enumerate}
\end{proposition}

Taking the previous proposition into account it is important to estimate how close typical points
return close to the critical region.

\subsection{Partial hyperbolicity and admissible curves}

Under our assumptions on the generalized Viana maps of Definition~\ref{def:generalized},
we obtain that the map is indeed partially hyperbolic in the sense that the dynamics along the 
horizontal direction dominates the dynamics along the vertical fibers. This will be made precise
in terms of admissible curves as we now describe.

\begin{definition}
A curve $\hat Y=\text{graph}(Y)$ with $Y: (0,1] \to I_0$ is an \emph{admissible curve}  if
it is $C^2$ differentiable, 
$|Y'(\theta)|\le \al$ and $|Y''(\theta)|\le \al$ for every $\theta\in (0,1]$.
\end{definition}

The strong expansion assumption on the Markov map $g$ yields a domination property as we now describe. 

\begin{lemma}
If $\hat Y$ is an admissible curve then for every $\om\in \cP^{(n)}$ it follows that $\varphi^n(\hat Y\mid_\om)$
is an admissible curve. In particular $\varphi^n(\hat Y)$ is an at most countable collection of 
admissible curves.
\end{lemma}

\begin{proof}
This lemma follows from \cite[Lemma~2.1]{Vi97}, whose argument we reproduce here for completeness and
the reader's convenience. Since it is enough to prove the lemma for $n=1$ and use the argument recursively, 
let $\hat Y=\text{graph}~{Y}$ be an admissible curve and take $\om\in\cP$. Then, for every
$\theta\in P$
$$
\varphi(\hat Y(\theta)) = \varphi(\theta, Y(\theta))
		= (g(\theta),f(\theta,Y(\theta)))
		= (g(\theta),Y_1(g(\theta)))
$$
where, by the chain rule and definition of $Y_1:(0,1] \to I$,
$$
|Y_1'(g(\theta))|  
	\leq \frac{1}{|g'(\theta)|} \, 
		\left|  \partial_\theta f (\hat Y(\theta)) +\partial_x f(\hat Y(\theta)) \, Y'(\theta) \right|
	\leq \frac{2\pi+4}{16} \al 
	<\al.
$$
Analogously it is not hard to check that
\begin{align*}
|Y_1''(g(\theta))|  
	\leq \frac{1}{|g'(\theta)|^2} \, &
		\left|  \partial_{\theta\theta} f (\hat Y(\theta)) +\partial_{x\theta} f(\hat Y(\theta)) \, Y'(\theta) 
			+ \partial_{\theta x} f(\hat Y(\theta)) \, Y'(\theta) \right. \\
			& \left. + \partial_{x} f(\hat Y(\theta)) \, Y''(\theta) +  \partial_{x x} f(\hat Y(\theta)) \, (Y'(\theta))^2 
			- Y_1'(g(\theta)) \, g''(\theta) \right|, 
\end{align*}
which is smaller than $\al$ since the partial derivatives of $f$ are smaller compared with the term 
$1/|g'|^2 \le 1/\,16^2$. 
This proves that $\hat Y_1 = \varphi(\hat Y\mid_\om)$ is an admissible curve.
Since $\cP$ is at most countable then $\varphi(\hat Y)$ is the union of at most countable admissible curves.
This finishes the proof of the lemma.
\end{proof}

The crucial property of admissible curves is that their images by $\varphi$ are non-flat. 

\begin{lemma}\label{le:transversality}
Let $\hat Y=\text{graph}~{Y}$ be an admissible curve and set $\hat Y_1(\theta)=\varphi(\hat Y(\theta))
=(g(\theta),Y_1(\theta))$. Then $|Y_1'(\theta)|\ge \al/2$ or $|Y_1''(\theta)|\ge 4\al$ for every $\theta\in (0,1]$ and
\begin{equation*}
\Leb \left(  \theta\in (0,1] :  \hat Y_1(\theta) \in (0,1]\times I \right) 
	\leq \frac{6|I|}{\al} + 2\sqrt{\frac{|I|}{\al}}.
\end{equation*}
for any interval $I\subset I_0$.
\end{lemma}

\begin{proof}
The proof follows the same ideas of \cite[Lemma~2.2]{Vi97} even with the presence of discontinuities for $g$. 
In fact, set $\hat Y(\theta)=(\theta,Y(\theta))$ and notice that for all $\theta\in (0,1]$
\begin{align}\label{eq:der1}
Y_1'(\theta) =\partial_\theta f (\hat Y(\theta)) +\partial_x f(\hat Y(\theta)) \, Y'(\theta) 
		  = 2\pi \alpha \cos (2\pi\theta) -2 Y(\theta) \, Y'(\theta)
\end{align} 
and $Y_1''(\theta)= - 4\pi^2 \sin(2\pi\theta) - 2 Y'(\theta)^2 -2 Y(\theta) \, Y''(\theta)$.
If $\theta\in \mathcal A=\{\tilde \theta \in (0,1]: |\sin(2\pi\tilde\theta)|\leq \frac13 \}$ then 
$|\cos (2\pi\theta)|\geq \frac{11}{12}$ and it follows from~\eqref{eq:der1} that
$|Y_1'(\theta)|\geq  (\frac{11\pi}{6} -4) \al>\frac{\al}{2}$. Otherwise, for $\theta\in (0,1]\setminus \cA$ 
it follows that $|Y_1''(\theta)|\geq  (\frac{4\pi^2}{3} -2 - 2 \al) \al\geq 4\al$. This proves the first statement of the lemma.

Now, notice that $\mathcal A$  has three connected components and 
$|Y_1'(\theta)|\geq \al/2$ for all $\theta\in \mathcal A$. Moreover, since the map $\theta\mapsto  Y_1(\theta)$ 
is $C^1$-differentiable then applying the Mean Value Theorem applied to $Y_1$ on each connected component of 
$\mathcal A$ it follows that
$
\Leb \left(  \theta\in \mathcal A :  \hat Y_1(\theta) \in (0,1]\times I \right)
	\leq \frac{6|I|}{\al}.
$ 
Using that $|Y_1''|\geq 4\al$ on the two connected components of $(0,1]\setminus \mathcal A$ a similar argument 
shows that
$
\Leb \left(  \theta\in (0,1]\setminus \mathcal A :  \hat Y_1(\theta) \in (0,1]\times I \right)
	\leq 4 \sqrt{\frac{|I|}{4\al}}=2 \sqrt{\frac{|I|}{\al}}.
$ 
The proof of the lemma is now complete. 
\end{proof}

We obtain the following very useful consequence.

\begin{corollary}\label{cor:iteration}
Let $\hat Y=\text{graph}~(Y)$ be an admissible curve and set $\hat Y_j=\varphi^j(\hat Y)$.
Then there exists $C>0$ (depending only on $g$) such that for any subinterval satisfying $|I|\le \al$
it holds 
\begin{equation*}
\Leb \left(  \theta\in (0,1] :  \hat Y_j(\theta) \in (0,1] \times I \right)
	\leq C \sqrt{\frac{|I|}{\al}}
	\quad\text{for all} \; j\ge 1.
\end{equation*}
\end{corollary}

\begin{proof}
First note that the case $j=1$ corresponds to the previous lemma. Hence, let $j\ge 2$ be arbitrary and fixed. 
If $\om\in \cP^{(j-1)}$ then $g^{j-1}(\om)=(0,1]$ and $\varphi^{j-1}(\hat Y\mid_\om)$ is an admissible curve. Therefore, 
$
\Leb (  \theta\in (0,1]:  \varphi(\varphi^{j-1}(\hat Y\mid_\om (\theta) )) \in (0,1]\times I )
	\leq 6\sqrt{\frac{|I|}{\al}}
$
for any subinterval $I$ satisfying $|I|\le \al$. Then one can use the bounded distortion property 
to get
\begin{align*}
\Leb \left(  \theta\in (0,1]:  \hat Y_j(\theta) \in \hat I \right)
	& = \sum_{\om\in \cP^{(j-1)}}
		\Leb \left(  \theta\in \om:  \varphi(\varphi^{j-1}(\hat Y\mid_\om (\theta) )) \in \hat I \right) \\
	& \leq \sum_{\om\in \cP^{(j-1)}} 6\frac{Kd}{d-1}\sqrt{\frac{|I|}{\al}} |\om| 
	\leq C\sqrt{\frac{|I|}{\al}},
\end{align*}
where $\hat I = (0,1] \times I$ and $C=6K>0$ depends only on $g$. The proof of the corollary is now complete.
\end{proof}

\begin{remark}
Let us mention that the bound in the right hand side of the expression in Corollary~\ref{cor:iteration} depends on 
$\al$  and it increases when $\al$ approaches zero. In fact, as $\al$ will be required to be very small 
for the large deviations argument in Subsection~\ref{subsec:largedeviations}, it is necessary to obtain similar estimates where the right hand side above does not depend on $\al$.
\end{remark}

\section{Positive Lyapunov exponents for the skew-product $\varphi_\al$}\label{sec:positive.exp}

In this section we study the recurrence of typical points near the critical region and deduce
existence of positive Lyapunov exponents Lebesgue almost everywhere.

\subsection{Recurrence estimates}

For notational simplicity set $\hat J(r)=(0,1]\times  J(r)$ for every $r$, with
$J(r)=[-\sqrt{\alpha} e^{-r}, \sqrt{\alpha} e^{-r}]$. In the next proposition we show that for a sufficiently 
large iterate $\varphi^{M(\al)}$ of $\varphi$ the measure of the set of points which have exponential 
deep returns decay exponentially fast with a rate that does not involve $\al$. More precisely, if 
$0<\eta<\frac13$ is as given in Propositon~\ref{prop:uniform.expansion} we obtain the following.

\begin{proposition}\label{prop:main1}
There exists  $\vep>0$ small so that every $\varphi$ that is $\vep$-$C^3$-close to $\varphi_\al$ satisfies
the following property: there exists $\tilde C,\beta>0$ and for any given $\al>0$ there is a positive 
integer $M=M(\al) < N(\al)$ such that if  $\hat Y = \text{graph}~(Y)$ is an admissible curve, then
\begin{equation*}
\Leb \left( \theta\in (0,1] : \hat Y_{M}(\theta) \in \hat J(r-2) \right)
	\leq \tilde C e^{-5\beta r}
\end{equation*}
for every $r\geq \left(\frac{1}{2}-2\eta\right)\log\frac{1}{\al}$.
\end{proposition}

\begin{proof}

Let $\hat Y$ be a fixed arbitrary admissible curve and set 
$\hat Y_j(\theta)=\varphi_\al^j(\hat Y (\theta)=(g^j(\theta),Y_j(\theta))$.
On the one hand using Corollary~\ref{cor:iteration} we deduce that
\begin{equation*}
\Leb \left(  \theta\in (0,1] :  \hat Y_j(\theta) \in \hat J(r-2) \right)
	\leq C \al^{-\frac12}\sqrt{|J(r-2)|}
	\leq C \al^{-\frac14} e^{-\frac12(r-2)},
\end{equation*}
which satisfies the assertion in the corollary 
with $\tilde C=Ce$ and $\beta=\frac1{15}$, provided that  $\frac{r}6 \geq \frac14 \log \frac1\al$. 
So, through the remaining we assume 
$\left(\frac{1}{2}-2\eta\right)\log\frac{1}{\al}\leq r \leq \frac32 \log \frac1\al$.  Let
$M=M(\al)$ be maximal such that $32^M\al \leq 1$, and note that $M< N$.
One can write
\begin{align*}
\Leb \left(  \theta\in (0,1] :  \hat Y_M(\theta) \in \hat J(r-2) \right)
	& = \sum_{\omega\in\cP^{(M)}} \Leb \left(  \theta\in \omega :  \hat Y_M(\theta) \in \hat J(r-2) \right) \\
	& = \sum_{\underline{s} \in S^M} 
			\Leb \left(  \theta\in \omega_{\underline{s}} :  \hat Y_M(\theta) \in \hat J(r-2) \right),
\end{align*}
where $\underline{s}=(s_1, s_2, \dots, s_M)\in S^M$ is the itinerary for the elements of $\cP^{(M)}$. 
The strategy is to subdivide  itineraries $\underline{s}=(s_1, s_2, \dots, s_M)\in S^M$ according to
its average depth $\sum_{i=1}^M s_i(\theta)$. On the other hand, we use a large deviations argument to
show that points with large average depth decrease exponentially fast. On the other hand, admissible curves 
associated to points with smaller average have vertical displacement.

\vspace{.2cm}
\noindent \emph{Claim:} There exists $\zeta\in(0,1)$ and for any admissible curve 
$\hat Z$ there are disjoint collections $\cP_1$ and $\cP_2$ of elements of $\cP$ such that the image 
$\hat Z_1(\theta)=\varphi(\hat Z)(\theta)=(g(\theta), Z_1(\theta))$ satisfies
$
|Z_1\mid_{\omega}-Z_1\mid_{\tilde\om}| \geq \frac{\al}{100}
$ 
for all $\omega\in\cP_1$ and $\tilde \om\in\cP_2$ 
and
$$
\zeta \leq \Leb\Big(\bigcup_{\omega\in \cP_1} \om \Big) 
	\leq \Leb\Big(\bigcup_{\omega\in \cP_2} \om \Big) 
	\leq 1-\zeta.
$$

\begin{proof}[Proof of the Claim:] 
It follows from Lemma~\ref{le:transversality} that the admissible curve $\hat Z_1$ has two critical points
$\theta'_1<\theta'_2$ one in each connected component of the set $(0,1]\setminus \cA=\{\theta\in(0,1] : |\sin(2\pi\theta)| > \frac13 \}$. Hence  $\frac{1}{2\pi}\arcsin(\frac13)<\theta'_1< \frac12 -\frac{1}{2\pi}\arcsin(\frac13)$  and also
$\frac12+ \frac{1}{2\pi}\arcsin(\frac13)<\theta'_2< 1 -\frac{1}{2\pi}\arcsin(\frac13)$.  

On the one hand, if $\theta_1',\theta_2' \notin [\frac14,\frac34]$ then $[\theta'_1,\theta'_1+\frac1{16})$ and  $(\theta'_2 -\frac1{16}, \theta'_2]$ are disjoint intervals that do not intersect the middle component $[\frac12 -\frac{1}{2\pi}\arcsin(\frac13), 
\frac12 +\frac{1}{2\pi}\arcsin(\frac13)]$ in $\cA$. Then, using that $\hat Z_1|_{[\theta'_1, \theta'_2]}$ is strictly 
monotone and $|Z_1'|_{\cA}|\geq \frac\al{2}$ 
$$
\inf Z_1|_{[\theta'_1,\theta'_1+\frac1{16}]} -\sup Z_1|_{[\theta'_2-\frac1{16}, \theta'_2]} 
	\geq \frac{\al}{2\pi}\arcsin(\frac13)
	\geq \frac{\al}{100}.
$$

On the other hand, if $\theta'_1\geq \frac14$ then $(0,\frac1{16}]$ and  $(\theta_1'-\frac1{16}, \theta_1']$ are 
disjoint intervals. Moreover, using that $Z_1|_{(0,\theta_1']}$ is strictly monotone and that 
$|Z_1''|_{(\frac{1}{16},\frac3{16}]}|\ge 4\al$ we obtain 
analogouly 
$
\inf Z_1|_{(\theta_1'-\frac1{16}, \theta_1']} -\sup Z_1|_{(0,\frac1{16}]} 
	\geq \frac{4\al}{64}
	\geq \frac{\al}{100}.
$
A similar reasoning holds for the case $\theta'_2\leq \frac34$. 

Since all partition elements of $\cP$ have length smaller or equal to $\frac1{16}$, there are 
collections $\cP'_1$ and $\cP'_2$ of elements in $\cP$ such that $\inf Z_1\mid_{\omega}-\sup Z_1\mid_{\tilde\om} 
\geq \frac{\al}{100}$ for all $\omega\in\cP'_1$ and $\tilde \om\in\cP'_2$.  In addition, using that 
$\cP_1'$ and $\cP'_2$ are disjoint we get
$$
1- \frac{1}{16}
	\geq \Leb\Big(\bigcup_{\omega\in \cP_i'} \om \Big)
	\geq \frac{1}{16}
$$
for $i=1,2$ and we set $\zeta=\frac{1}{16}$. Finally if
$\Leb(\bigcup_{\omega\in \cP_1} \om ) \leq \Leb(\bigcup_{\omega\in \cP_2} \om)$ we define
$(\cP_1,\cP_2):=(\cP'_1,\cP'_2)$ while otherwise $(\cP_1,\cP_2):=(\cP'_2,\cP'_1)$.
This finishes the proof of our claim.
\end{proof}

Let $\om,\tilde\om \in \cP^{(M)}$ be arbitrary with $\iota_M(\om)=(s_1, \dots, s_M)$ and $\iota_M(\tilde\om) =(\tilde s_1, \dots, \tilde s_M)$, and assume that $|Y_M(\theta)|<\sqrt{\al}$ for some $\theta$ since otherwise there is 
nothing to prove.  
We collect some facts whose proof can be found in~\cite[p.72-73]{Vi97}:
\begin{enumerate} 
\item (Expansion estimates) Given an arbitrary $\hat y\in \hat Y$ and $0\le j \leq M-1$
	$$
	\la_j:=|\partial_x f^{M-j} (\varphi_\al^j(\hat y))| 
		\geq C_2 \si_2^{M-j}
	$$
\item (Bounded distortion) For all $0\le j \le M-1$, all $(\theta_j,x_j) \in \hat Y_j$ and $1\le i \le M-j$
	it holds that
	$$
	\frac12 \frac{\la_j}{\la_{i+j}}
	\leq |\partial_x f^{i} (\theta_j,x_j) | 
	\leq 2 \frac{\la_j}{\la_{i+j}}
	$$
\item (Positive frequency) If $K=400 e^2$, $t_1=1$ and define recursively 
	$t_{i+1}=\min \{t_i<s\le M : \la_{t_i} \ge 2K \la_s\}$ then there exists $\ga_1>0$ (depending only
	on $\eta$) such that $k(r)=\max\{i : \la_{t_i} \ge 2\al^{-\frac12} Ke ^{-r}\}$ satisfies $k(r)\ge \ga_1 r$.
	\vspace{.2cm}
\item (Vertical displacement) For any $1\le i \le k$ and $(s_1, \dots, s_{t_{i}-1})$ there are collections
	$\cP_{1,i}$ and $\cP_{2,i}$ as in Claim~1 (with corresponding symbols $s^1_{t_i}, s^2_{t_i}$) for 
	which the admissible curves $\varphi^{t_i}(\om_{s_1, \dots, s^1_{t_{i}}}) $ and 
	$\varphi^{t_i}(\om_{s_1, \dots, s^2_{t_{i}}})$ satisfy
	$$
|\varphi^{t_i}(\om_{s_1, \dots, s_{t_i}-1, s^1_{t_{i}}})(\theta) - \varphi^{t_i}(\om_{s_1, \dots, s_{t_i}-1, s^2_{t_{i}}}(\theta)|
		\geq \frac{\al}{100}  
		\quad \text{ for all } \theta\in(0,1]	
	$$
	and, consequently, for any $(s_{t_i}+1, \dots, s_M) \in S^{M-t_i}$ 
	\begin{align*}
	|\varphi^{M}(\om_{s_1, \dots, s_{t_i}-1, s^1_{t_{i}},  s_{t_i}+1, \dots, s_M })(\theta) 
		 - \varphi^{M}  (\om_{s_1, \dots, s_{t_i}-1, s^1_{t_{i}}, s_{t_i}+1, \dots, s_M })(\theta)|  
		 \geq 4 \sqrt{\alpha} e^{-(r-2)}.
	\end{align*}
\end{enumerate}

Now we are in a position to finish the proof of Proposition~\ref{prop:main1}. By a small abuse of notation,
we write $s_{t_i}(\theta)\in\cP_{1,i}$ meaning that $\om_{s_{t_i}(\theta)} \in \cP_{1,i}$, and analogously for the
collection $\cP_{2,i}$.
In fact, one can combine Claim~1 with property (4) above  to show that 
$\Leb \left(  \theta\in (0,1] :  \hat Y_M(\theta) \notin \hat J(r-2) \right)$ is given by
\begin{align*}
\sum_{s_1 \in S}  & \sum_{(s_2,\dots, s_{M-1})\in S^{M-1}} 
			\Leb \left(  \theta\in \om_{s_1, \dots, s_M} :  \hat Y_M(\theta) \notin \hat J(r-2) \right) \\
	& \geq  \Leb \left( \theta :s_1(\theta) \in\cP_{1,1} \right) 
	 + \Leb \left( \theta :s_1(\theta) \notin\cP_{1,1} \text{ and } \hat Y_M(\theta) \notin \hat J(r-2) \right),
\end{align*}
since admissible segments over $\cP_{1,1}$ and $\cP_{2,1}$ correspond to vertically displaced
admissible curves when mapped by $\varphi^M$ and $\mu(\bigcup_{\omega\in \cP_{1,1}} \om ) \leq 
\mu(\bigcup_{\omega\in \cP_{2,1}} \om)$.
Analogously, the second term in the right hand-side above satisfies
\begin{align*}
\Leb ( \theta  \in (0,1]:\, &  s_1(\theta)  \notin\cP_{1,1} \text{ and } \hat Y_M(\theta) \notin \hat J(r-2) ) \\
	& = \sum_{s_1 \notin \cP_{1,1}}  \sum_{s_2,\dots, s_{t_2-1}} \sum_{ s_{t_2}\notin\cP_{1,2}} 
		\sum_{s_{t_2+1},\dots,s_M} 
		\Leb ( \theta : 
			\hat Y_M(\theta) \notin \hat J(r-2) ) \\
	& + \sum_{s_1 \notin \cP_{1,1}}  \sum_{s_2,\dots, s_{t_2-1}} \sum_{ s_{t_2}\in\cP_{1,2}} 
		\sum_{s_{t_2+1},\dots,s_M} 
		\Leb ( \theta : 
			\hat Y_M(\theta) \notin \hat J(r-2) ) \\
	& \geq \sum_{s_1 \notin \cP_{1,1}}  \sum_{ s_{t_2}\notin\cP_{1,2}} 
		\Leb ( \theta : s_1(\theta)  \notin\cP_{1,1} \text{ and } s_{t_2}(\theta)  \notin\cP_{1,2}) \\
	& + \Leb ( \theta : s_1(\theta)  \notin\cP_{1,1} \text{ and } s_{t_2}(\theta)  \in\cP_{1,2} \text{ and } \hat Y_M(\theta) 
	\notin \hat J(r-2)) 
\end{align*}
Proceeding recursively we obtain that
\begin{align*}
\Leb \big(  \theta :  \hat Y_M(\theta) & \notin  \hat J(r-2) \big) \\
		& \geq \Leb \left( \theta :s_1(\theta) \in\cP_{1,1} \right) 
		 + \Leb \left( \theta :s_1(\theta) \notin\cP_{1,1} \text{ and } s_{t_2}(\theta) \in\cP_{1,2} \right) \\
		& + \Leb \left( \theta :s_1(\theta) \notin\cP_{1,1} \text{ and } s_{t_2}(\theta) \notin\cP_{1,2} \text{ and } 
		s_{t_3}(\theta) \in\cP_{1,3}
			\right) \\
		& + \dots \\
		& + \Leb \left( \theta :s_{t_i}(\theta) \notin\cP_{1,i}, \forall 1\leq i \leq k(r)-1 \text{ and } s_{t_{k(r)}}(\theta) \in\cP_{1,k} 
			\right),
\end{align*}
proving 
$\Leb \left(  \theta\in (0,1] :  \hat Y_M(\theta) \in \hat J(r-2) \right) 
	 \leq \Leb \left(  \theta :  s_{t_i}(\theta) \notin\cP_{1,i}, \forall 1\leq i \leq k(r) \right)$.
Hence, to finish the proof it is enough to prove that the previous right hand side decreases 
exponentially fast on $r$. 
Since $\varphi(\theta,x)=(\tilde g(\theta),\tilde f(\theta,x))$ where $\|g-\tilde g\|_{C^3}<\vep$ and 
$g$ is piecewise linear satisfying $|g'|\ge d$ then it follows from 
Lemmas~\ref{lem:open} and~\ref{lem:bdistortion}  that
\begin{align*}
\frac{\Leb(\om_{(s_1,s_2, \dots, s_{k+\ell})})}
	{\Leb(\om_{(s_1,s_2, \dots, s_k)})}
	&\leq
	\frac{\Leb(\tilde g^k(\om_{(s_1,s_2, \dots, s_{k+\ell})}))}
	{\Leb(\tilde g^k(\om_{(s_1,s_2, \dots, s_k)}))} 
	\exp\Big(  \frac{d\vep}{(d-1)(d-\vep)^2}  \Big)^2 \\
	& = \Leb(\om_{(s_{k+1},s_{k+2}, \dots, s_{k+\ell})})
	\exp\Big(  \frac{d\vep}{(d-1)(d-\vep)^2}  \Big)^2
\end{align*} 
and consequently
\begin{align*}
\Leb(\om_{(s_1, \dots, s_{k+\ell})})
	\leq \Leb(\om_{(s_1,\dots, s_k)}) \Leb(\om_{(s_{k+1}, \dots, s_{k+\ell})})
		\exp\Big(  \frac{d\vep}{(d-1)(d-\vep)^2}  \Big)^2
\end{align*} 
for any $k+\ell \ge 1$ and sequence $(s_1,s_2, \dots, s_{k+\ell})\in S^{k+\ell}$. Thus if
$\vep$ is small, the previous estimates together with $k(r)\geq \ga_1 r$ yield that
\begin{align*}
\Leb \left(  \theta :  s_{t_i}(\theta) \notin\cP_{1,i}, \forall 1\leq i \leq k(r) \right)  
	& =\sum_{\{(s_1,s_2, \dots, s_M) : s_{t_i} \notin\cP_{1,i} \}}   
			\Leb(\om_{(s_1,s_2, \dots, s_M)}) \\
	& \leq \exp\Big(  \frac{d\vep}{(d-1)(d-\vep)^2}  \Big)^{2k(r)} \prod_{j=1}^{k(r)} \left[ 1- \Leb(\cP_{1,i}) \right] \\
	&  \leq \left[\exp\Big(  \frac{2d\vep}{(d-1)(d-\vep)^2}  \Big) (1-\zeta)\right]^{k(r)} \\
	& \leq \left[\exp\Big(  \frac{2d\vep}{(d-1)(d-\vep)^2}  \Big) (1-\zeta)\right]^{\ga_1 r}.
\end{align*}
This proves that  $\Leb \left(  \theta\in (0,1] :  \hat Y_M(\theta) \in \hat J(r-2) \right)$ decreases exponentially fast
in $r$ provided that $\vep>0$ is small, and finishes the proof of the proposition.
\end{proof}

Observe that our hypothesis yield that the base map is (semi)-conjugated to the on-sided shift with 
alphabet $S$. Let us point out that a much simpler argument would lead the same result above provided 
if the $\mu$ is Bernoulli, that is, defined by mass distribution.

\subsection{Positive Lyapunov exponents}\label{subsec:largedeviations}

We are now in a position to prove Theorem~\ref{thm:Main.Estimate} similarly to \cite{Vi97}. 
First we consider the skew-product $\varphi_\al$.
Let $\ga\in(0,1)$ be arbitrary and fixed.
For any integer $n\ge 1$ set $m=[\sqrt{n}]$ and $\ell=m-M$, where $[\cdot]$ stands as before for 
the integer part. Given an admissible curve $\hat Y$ and $v\in \mathbb R^2$ non-colinear with $\partial/\partial x$
there is $C>0$ so that 
$
\| D\varphi^n_\al(\hat Y(\theta)) v \|
	 \geq C |(g^n)'(\theta)| 
	 \geq C d^n 
$
grows exponentially fast. Hence, it remains to estimate the derivative
\begin{align*}
 \left\| D\varphi^n_\al(\hat Y(\theta))\frac{\partial}{\partial x} \right\|
	 = \prod_{j=0}^{n-1} \left| \frac{\partial f}{\partial x} (\hat Y_j (\theta)) \right|, 
\end{align*}
where the later product can be estimated according to the returns near the critical region using
Proposition~\ref{prop:uniform.expansion}. 
We will say that $1\leq \nu \leq n$ is a \emph{deep return for $\theta$}
if $\theta\in \om$ for some partition element $\om\in\cP^{(\nu+\ell)}$ satisfying 
$\varphi^\nu(\hat Y\mid_\om) \cap ((0,1] \times J( m)) \neq \emptyset$.
We will say that $1\leq \nu \leq n$ is a \emph{regular return for $\theta$} if
$\theta\in \om$ where $\om\in\cP^{(\nu+\ell)}$ satisfies
$\varphi^\nu(\hat Y\mid_\om) \cap ((0,1] \times J(0))\neq \emptyset$ and
$\varphi^\nu(\hat Y\mid_\om) \cap ((0,1] \times J( m))= \emptyset$. In this case 
we set the \emph{return depth} 
$
r_\nu(\theta)=\min\{r < m : \varphi^\nu(\hat Y\mid_\om) \cap ((0,1] \times J(r))\neq \emptyset\}.
$
Observe that the function $r_\nu(\cdot)$ is constant on the elements of the partition $\cP^{(\nu+\ell)}$.
Moreover, since $\varphi^\nu(\hat Y\mid_\om)$ is a curve with slope smaller or equal to $\al$ and 
horizontal length smaller or equal to $16^{-\frac{\sqrt{n}}2}$, if $\nu$ is a deep return 
then $\varphi^\nu(\hat Y\mid_\om) \subset ((0,1] \times J( m-1))$. In consequence,
from Corollary~\ref{cor:iteration} we get
\begin{align*}
\Leb (\theta\in & (0,1] :  \exists 1\leq \nu \leq n \text{ deep return for } \theta) \nonumber \\
	& \leq n \, \Leb(\theta\in (0,1] :  \hat Y_\nu(\theta) \in ((0,1] \times J(m-1))) \nonumber \\
	& \leq n\, C \al^{-\frac14} e^{-\frac{m}2} \\
	& \leq \al^{-\frac14} e^{-\frac13 \sqrt{n}} \nonumber 
\end{align*}
for all large $n$. In addition, if $\theta\in (0,1]$ has no deep returns and $1\leq \nu_1 < \nu_2 < \dots <\nu_s \leq n$
are the regular returns for $\theta$ with return depths $r_1, \dots, r_s$ respectively, then 
it follows from the estimates in \cite[p. 76]{Vi97} that for all large $n$
\begin{equation}\label{eq:Lyapunov.exp}
\log \left\| D\varphi^n_\al(\hat Y(\theta))\frac{\partial}{\partial x} \right\|
	\geq 2c n -\sum_{i\in G_\theta} r_i(\theta),
\end{equation}
where $c=\frac13 \min\{\gamma_2,\log\si_2\}>0$ and 
$G_\theta=\{1\leq i \leq s : r_{\nu_i} (\theta) \geq (\frac12 -\eta)\log \frac1\alpha\}$.
Therefore, if $G_\theta(q)=\{i : \nu_i \equiv q (\!\!\!\mod m)\} $ (for $0\leq q<m$) we obtain that
\begin{align*}
\Leb\left(\theta\in (0,1] :  \sum_{i\in G_\theta} r_i(\theta) \geq cn\right) 
	& \leq \sqrt{n} \Leb\left(\theta\in (0,1] : \sum_{i\in G_\theta(q)} r_i(\theta) \geq \frac{cn}m\right) \nonumber\\
	& = \sqrt{n} \sum_{R\geq \frac{cn}m} \sum_{\substack{(\rho_1, \dots, \rho_\tau)\\ \sum \rho_j =R}} 
		\Leb\left(\theta : r_{\nu_j}(\theta) =\rho_j, \; \forall j \right) \nonumber\\
	& \leq \sqrt{n} \sum_{R\geq \frac{cn}m} \left( \begin{array}{c} R+\tau \\ \tau \end{array} \right) 
		\tilde C^{\tau} e^{-5\beta \sum_j \rho_j},
\end{align*}
where $\tau$ denotes the number of nonzero depths $r_j$. Using that $R\geq \tau (\frac12-\eta)\log\frac1\al$
we can 
take $\al$ small so that  $\frac{ (R+\tau)!}{R! \, \tau!} \, \tilde C^{\tau} \leq e^{\beta R}$ and since 
\begin{equation*}\label{eq:tail}
\Leb\Big(\theta\in (0,1] :  \sum_{i\in G_\theta} r_i(\theta) \geq cn\Big) 
	\leq \sqrt{n} \sum_{R\geq c \,\frac{n}m}  e^{-4\beta R}
	\leq e^{-\beta \sqrt{n}}
\end{equation*}
for all large $n$, it is summable. Borel-Cantelli lemma yields that for Lebesgue almost every $\theta$ the
expression $\sum_{i\in G_\theta} r_i(\theta) \le cn$ holds for all but finitely many values of $n$.
Together with equation~\eqref{eq:Lyapunov.exp} above and using that $\hat Y$ was chosen arbitrary,  
this proves that $\varphi_\al$ has only positive Lyapunov exponents, that is,
$$
\liminf _{n\to\infty}\frac1n \log \|D\varphi_\al^n(\theta,x) v \|\geq c 
$$
for Lebesgue almost every $(\theta,x)$ and every $v\in \mathbb R^2$, 
proving the assertion of Theorem~\ref{thm:Main.Estimate} for the skew-product $\varphi_\al$.
Now, let $\varphi$ be $\vep$-$C^3$-close to $\varphi_\al$. For $\varphi$ let the critical region $\mathcal C$
be defined by $(\theta,x) \in \mathcal C$ if and only $\det D\varphi_\al(\theta,x)=0$. It is not hard 
to deduce from the Implicit Function Theorem that $\mathcal C$ is a $C^2$-smooth curve
on each invertibility domain, that is, there exists a function $\eta: (0,1] \to I_0$ that is $C^2$-close to 
zero on each interval $\omega_i$ satisfying $\mathcal C=\text{graph}(\eta)$.
Thus one can make a $C^2$ change of coordinates and assume that the critical region 
$\mathcal C$ coincides with the segment $\{x=0\}$. Moreover, since $\partial_x f(\theta x)=2x$ 
we may assume $\partial_x \tilde f(\theta,x)=|x| \psi(\theta,x)$ with $\psi$ close to $2$, behaves like a power 
of the distance to the critical region. This allows to reproduce the previous argument and to show that $\varphi$
has two positive Lyapunov exponents and finishes the proof of Theorem~\ref{thm:Main.Estimate}.

\section{{SRB} measures and their statistical properties}\label{sec:proofs}

This section is devoted to the study of ergodic properties of these robust nonuniformly expanding transformations. 
In fact we show that there is a unique SRB measure and prove that it has good statistical properties. Throughout
let $\varphi$ be $C^3$-close to $\varphi_\al$ and let $\La$ denote the corresponding attractor.
We will say that $\varphi$ is  \emph{topologically exact} if for any open set $U$ there exists $N=N(U)\ge 1$ 
such that $\varphi^N(U)\supset \La$. We say that $\varphi$ is \emph{ergodic} (with respect to Lebesgue) 
if all $f$-invariant measurable sets are zero or full Lebesgue measure sets. 

\begin{proposition}\label{prop:exact}
The map $\varphi$ is topologically exact and ergodic with respect to $\Leb$.
\end{proposition}

\begin{proof}
 Since the proof follows closely \cite[Theorem~C]{AV02} we will omit the details.
\end{proof}

At this point one could use \cite{ArS11} to obtain the existence of the absolutely continuous
invariant probability measure. However, to deduce the good statistical properties in 
Theorem~\ref{thm:SRBmeasure} and to deal with the critical set $\mathcal C$ and 
discontinuities $\mathcal D$ for $\varphi$ (formed by countable vertical segments) 
we need to estimate the tail of the hyperbolic times c.f. \cite[Definition~2.5]{Al00}. 
Since the arguments follow some now standard arguments we focus on the main ingredients.

Using that $\varphi$ is $C^3$-close to $\varphi_\al$, the critical region $\mathcal C$
obtained is a piecewise $C^2$-smooth curve where $\mathcal C=\{(\theta,x) : \det D\varphi(\theta,x)=0\}$.
Hence we may assume that $\mathcal C$ coincides with the segment $\{x=0\}$
and that $\varphi$ behaves like a power of a distance to the critical region along the invariant vertical foliation: 
there exists $B\ge 1, \be>0$ so that for all $(\theta,x)\in(0,1]\times I_0$ and all $v\in \mathbb R^2$
\begin{itemize}
\item[(C1)]  $\frac1B \text{dist}((\theta,x),\cC) \le \frac{\|D\varphi(\theta,x) v\|}{\|v\|}$
\end{itemize}
and for all points with  
$\dist((\theta_1,x_1), (\theta_2,x_2))<\text{dist}((\theta_1,x_1), \cC)/2$
\begin{itemize}
\item[(C2)] $|\; \log \|D\varphi(\theta_1,x_1)^{-1}\| - \log \|D\varphi(\theta_2,x_2)^{-1}\| \; |
			\leq B \frac{\dist((\theta_1,x_1), (\theta_2,x_2))}{\text{dist}((\theta_1,x_1), \cC)^\beta}$
\item[(C3)] $|\; \log |\det D\varphi(\theta_1,x_1) - \log |\det D\varphi(\theta_2,x_2) | \; |
			\leq B \frac{\dist((\theta_1,x_1), (\theta_2,x_2))}{\text{dist}((\theta_1,x_1), \cC)^\beta}$.
\end{itemize}
This will be used to control recurrence to the critical region $\mathcal C$. A first step to deduce stretched-exponential 
decay of correlations using the machinery developed in \cite{You98,ALP05,Gou06} we need to obtain non-uniform
expansion together with slow recurrence condition to both the critical region $\mathcal C$. Notice that 
\begin{align}\label{eq:nue1}\tag{NUE}
\liminf_{n\to\infty}  \frac1n \sum_{j=0}^{n-1} \log \|D\varphi(\varphi^j(\theta,x))^{-1} \|^{-1}
	 = \liminf_{n\to\infty}  \frac1n \log \big\|D\varphi^n(\theta,x) \frac{\partial}{\partial x} \big\|
	 \ge c>0 
\end{align}
and 
\begin{equation}\label{eq:nue2}\tag{SR}
(\forall \vep>0) \; (\exists\de>0) \quad
\limsup_{n\to\infty} \frac1n \sum_{j=0}^{n-1} -\log \text{dist}_\de(\varphi^j(\theta,x), \cC) <\vep
\end{equation}
holds for Lebesgue almost every $(\theta,x)$, where $\text{dist}_\de(z,\cC)=\text{dist}(z,\cC)$ 
if $\text{dist}(z,\cC)<\de$ and $\text{dist}_\de(z,\cC)=0$ otherwise. 
In fact, on the one hand
\begin{align*}
\cE_v(\theta,x)
	& :=\min\Big\{N\ge 1 : \frac1n \sum_{j=0}^{n-1} \log \|D\varphi(\varphi^j(\theta,x))^{-1} \|^{-1} \ge c, 
		\text{ for all }  n\ge N\Big\} \\
	& \le \min \Big\{N\ge 1 : r_j (\theta,x) \le c n 
		\; (\forall 1\le j \le n)   
		\text{ and }   \sum_{j=0}^{n-1} r_i(\theta,x) \le c n 
		\; (\forall n\ge N) \Big\},
\end{align*}
is well defined and finite for Lebesgue almost every point. Furthermore, it follows that 
$
\Leb( (\theta,x) \in(0,1]\times I_0 : \cE_v(\theta,x)\ge n) 
	\leq C e^{-\ga \sqrt{n}}
$
for all large $n$. On the other hand, given $\vep,\de>0$ consider the Lebesgue almost everywhere well defined function
\begin{align*}
\cR_{v,\vep,\de}(\theta,x)
	& :=\min\Big\{N\ge 1 : \sum_{j=0}^{n-1} -\log \text{dist}_\de(\varphi^j(\theta,x), \cC) <\vep \, n 
		\text{ for all }  n\ge N\Big\}.
\end{align*}
Notice that if $\de=(\frac12-2\eta)\log\frac1\al$ then 
$\sum_{j=0}^{n-1} -\log\text{dist}_\de(\varphi^j(\theta,x), \cC) \le \sum_{j=0}^{n-1} r_j(\theta,x)$
for all $(\theta,x)$. Hence, the same large deviations argument of Section~\ref{sec:positive.exp} yield that there exists  $\ga(\vep)>0$ such that
\begin{align*}
\Leb( (\theta,x) \colon \cR_{v,\vep,\de}(\theta,x)\ge n) 
	 \le \Leb\Big(\theta,x) \colon \sum_{j=0}^{n-1} r_j(\theta,x)>\vep n \Big)
	 \leq C e^{-\ga(\vep) \sqrt{n}}
\end{align*}
for all large $n$. In consequence, this proves that for any $\vep,\de>0$ there exists 
$\tilde \ga(\vep)=\min\{\ga,\ga(\vep)\}$ such that
$$
\Leb( (\theta,x) \in(0,1]\times I_0 : \cE_v(\theta,x)\ge n \text{ or } \cR_{v,\vep,\de}(\theta,x)\ge n ) 
	\leq C e^{-\tilde \ga(\vep) \sqrt{n}}
$$
for all large $n$.  On the one hand, despite the discontinuities, it follows from the Markov assumption and
bounded distortion Lemma~\ref{lem:bdistortion} that for any partition element $\om\in \cP^{n}$ the map
$g^n\mid_\om: \om \to (0,1]$ has bounded distortion and the backward contraction property
$
d( g^{n-j} (y), g^{n-j} (z)) \le (d-\vep)^{-j} d(g^n(y), g^n(z))
$
for all $0\le j\le n$ and $y,z\in \om$. In consequence, $g$ admits a unique absolutely continuous
invariant probability measure $\eta$ equivalent to Lebesgue.
Therefore, since $\varphi$ is a skew product, if
$
\cR_{h,\vep,\de}(\theta,x)
	 \!\!=\!\!\min\big\{N\ge 1 : \sum_{j=0}^{n-1} -\log \text{dist}_\de(\varphi^j(\theta,x), \mathcal D) <\vep \, n 
		\text{ for all }  n\ge N\big\}
$
and also
$
\cR_{g,\vep,\de}(\theta)
	 \!\!=\!\!\min\big\{N\ge 1 : \sum_{j=0}^{n-1} -\log \text{dist}_\de(g^j(\theta), \pi_1(\mathcal D)) <\vep \, n 
		\text{ for all }  n\ge N\big\}
$
then it follows that
\begin{align}\label{eq:deviations}
\Leb( (\theta,x) : \cR_{h,\vep,\de}(\theta,x) >n  )
	& = \Leb ( \theta\in (0,1] :  \cR_{g,\vep,\de}(\theta)>n)
\end{align}
where $\pi_1$ stands for the projection on the $\theta$-coordinate.
Furthermore, since $\eta$ is equivalent to Lebesgue and has exponential decay of correlations
it yields exponential large deviations for H\"older observables.  Thus, 
we can proceed as in \cite{AS13} (building over \cite{ALFV11}) to obtain an exponential large deviation
estimate for \eqref{eq:deviations} and, in consequence, sub-exponential tails for the expansion and 
return time functions. 
Thus, it follows from \cite{Gou06} that there exists a $\varphi$-invariant 
absolutely continuous probability measure $\mu$ with stretched-exponential decay of correlations.
In consequence,

\begin{proposition}\label{prop:existence.uniqueness}
There exists a unique SRB measure $\mu$ for $\varphi$. Moreover, the basin of attraction $B(\mu)$ contains Lebesgue almost every point in $\La$.
\end{proposition}

\begin{proof}
Let $\mu$ be the $\varphi$-invariant and ergodic probability measure constructed above. Since 
$\mu\ll \Leb$ then it is clearly an SRB measure because its basin of attraction
$$
B(\mu)=\Big\{(\theta,x)\in \La : \frac1n \sum_{j=0}^{n-1} \de_{\varphi^j(\theta,x)} \to \mu\Big\}
$$
is an $\varphi$-invariant set and a $\mu$-full measure set. Using that $\varphi$ is exact
then it follows that $B(\mu)$ has full Lebesgue measure set in $\La$. This also proves uniqueness 
of the SRB measure and finishes the proof of the proposition.
\end{proof}

Finally, it follows from \cite{ALFV11} that stretched-exponential decay of correlations imply on
stretched-exponential large-deviation bounds and also in the central limit theorem, almost sure invariance
principle, local limit theorem and Berry-Esseen theorem. This finishes the proof of Theorem~\ref{thm:SRBmeasure}.

\section{Coexistence of positive and negative Lyapunov exponents for generic generalized Viana-maps}\label{sec:periodic}

This section is devoted to the proof of Theorem~\ref{thm:generic.skew} where, in particular, we prove that 
generic generalized Viana maps have a dense set of points with negative central Lyapunov exponent. We 
will make use of the density of hyperbolicity for maps of the interval in Theorem~\ref{thm:KSV}.

\begin{proof}[Proof of Theorem~\ref{thm:generic.skew}]
Let $\mathcal V$ be an open set of generalized Viana maps.  By construction every $\varphi\in\mathcal V$ is conjugated to a skew-product over the same (topological) Markov expanding map $g$. Moreover, using that 
$g$ has a fixed point $p$ it follows that the map $\varphi$ preserves the fiber over one fixed point $p_*=p_*(\varphi)$.
In other words, one can consider the unimodal map 
$$
\varphi_{p_*}(\cdot)= \varphi(p_*,\cdot) : \{p_*\} \times I_0 \to \{p_*\} \times I_0.
$$
By topological conjugacy the continuation of the point $p_*(\varphi)$ is well defined for all $\varphi$.
Using Theorem~\ref{thm:KSV}, we deduce that there exists an open and dense set  $\cA\subset \mathcal V$ 
such that for any $\varphi\in \cA$ the corresponding unimodal map $\varphi_{p_*}$ is hyperbolic.
In consequence, $\varphi_{p_*}$ admits a periodic attractor whose the basin of attraction s a full Lebesgue 
measure set, it is  open and dense. Therefore, $\varphi$ has a periodic saddle $P_*$ on the fiber over $p_*$.

Now, notice that the pre-orbit $\cO^{-}(p_*)=\{ \theta\in (0,1] : (\exists n\ge 1) \; g^{n}(\theta)=p_* \}$ is dense
in $(0,1]$ and so the set $D= \{(\theta,x) : \theta\in \cO^{-}(p_*), x\in I_0  \}$ is dense in the attractor 
$\Lambda(\varphi)$.
Since the fibered maps are non-singular with respect to Lebesgue then there exists a dense subset of points
in $D$ that are forward asymptotic to the saddle point $P_*$ and, consequently, have one negative Lyapunov exponent as claimed in part (1). Since the second assertion follows directly from  Theorem~A this finishes the 
proof of the theorem. 
\end{proof}

\section{SRB measures for skew-products of fibered nearby hyperbolic interval maps}\label{sec:SRBs}

The main purpose of this section is to prove Theorem~\ref{thm:hiperbolicidade.fibrada} for fibered maps
obtained by perturbation of hyperbolic interval maps, which is of independent interest. So, let $T: I \to I$ be a  hyperbolic $C^3$-interval map with negative Schwarzian derivative in the interval $I$ and 
 $S:X \to X$ be a continuous map that admits a unique SRB measure $\mu_S$, where  $X$ is a compact 
 Riemannian manifold. Consider the skew product 
$$
\psi : (x,y)\mapsto (S(x), T(y)+a(x))
$$
associated to a $C^3$-function $a: X \to \mathbb R$.  Given a point $x\in X$, consider the iteration 
$\psi_x^n=\psi_{T^{n-1}(x)} \circ\dots\circ \psi_{T(x)}\circ \psi_x$. 
First we proceed to prove that if $\|a\|_{C^3}<\vep$, for a small $\vep$, then the skew product $\psi$ admits a
trapping region $\cU$. Moreover 
we prove that the attractor $\mathcal G\subset \cU$ supports a unique SRB measure 
$\nu$, ergodic,  whose basin of attraction contains Lebesgue almost every
point in $\cU$.

By hyperbolicity,  $T$ has a a finite number of periodic attracting points $\{p_i\}$ of period 
$\pi(p_i)\ge 1$ and $I=K\cup (\cup_i B(p_i))$, where $B(p_i)$ denotes the topological basin of attraction of $p_i$ 
and $K$ is an invariant Cantor set such  that $T\mid_K$ is expanding.
Since hyperbolicity is $C^r$-open and dense in the interval (see~\cite{KSvS07a})  and $\al$ is assumed
small, by structural stability all interval maps $T_x=T+a(x)$ are hyperbolic and topologically 
conjugated to $T$: for every $x\in X$ there exists an homeomorphism  $h_x$ that is $C^0$-close to identity,  varies continuously with $x$ and $T_{x} \circ h_x = h_x \circ T$. 
In consequence,  $p_i^x=h_x(p_i)$ is an attracting periodic point of period $\pi(p_i)$ and 
$K_x=h_x(K)$ is an invariant expanding Cantor set for the interval map $T_x$.
We first prove that similar features are inherited by the skew-product $\psi$.

\begin{proposition}\label{prop:hyp.saddle}
If $\vep>0$ is small, there exists $N\ge 1$ and an open set $U\subset I$ obtained as finite
union of open intervals and $c>0$ so that $\cU=X \times U$ is positively $\psi^N$-invariant and  
	$$
	\limsup_{n\to\infty} \frac1n \log \left\| D\psi^n(x,y) \frac{\partial}{\partial y} \right\| 
		\leq -c <0
	$$
for all $(x,y)\in \cU$.
Moreover, there exists a small open neighbourhood $V$ of the critical points for $T$ and $K\ge 1$ such that   
$\varphi^K( X \times V) \subset \mathcal U$.
\end{proposition}

\begin{proof}
Given any $i$, since $p_i$ is a periodic atractor for $T$ then using the uniform continuity of the derivative 
there exists $\hat \lambda_i<1$ and an open neighborhood $\cU_i$ of the forward orbit of $p_i$ such that $T^{\pi(p_i)}(U_i) \subsetneq U_i$ and $|(T^{\pi(p_i)})'(y)| \leq \hat \lambda_i^{\pi(p_i)}<1$ for all $y\in U_i$. 
Furthermore, if $\vep$ is small, the fact that $\|T-T_x\|_{C^3}<\vep$ there exists $\hat \lambda_i<\lambda_i<1$ satisfying  $T_x^{\pi(p_i)}(U_i) \subsetneq U_i$ and $|(T_x^{\pi(p_i)})'(y)| \leq \lambda_i^{\pi(p_i)}<1$ for all $y\in U_i$. 

Set $N=\prod_i \pi(p_i)$. On the one hand the later proves that the open set $U$ obtained by the finite union of the
open intervals $U_i$ is $T_x^N$-invariant for all $x\in X$ and, consequently, the set $\cU=X\times U$ is positively 
$\psi^N$-invariant. 
On the other hand, if  $\lambda$ is chosen such that $\la_i<\lambda<1$ for all $i$ then by the chain rule 
$
|(T_x^N)'(y)| 
	\leq \lambda^N
$
for all $y\in U$ and $x\in X$. Now we observe that if $n=qN+r$ with $q\in \mathbb N_0$ and $0\le r\le N-1$
then
\begin{align*}
\left\| D\psi^n(x,y) \frac{\partial}{\partial y} \right\|
	& = |(T_{S^{n-1}(x)}\circ \dots \circ T_{S(x)}\circ T_x)'(y)| \\
	& =  |(T_{S^{qN}(x)}^r)' (T_x^{qN} (y)) | \; \prod_{j=0}^{q-1}  | (T_{S^{jN}(x)}^N)' (T_x^{j N}(y)) | \\
	& \leq \|T'\|_{\infty}^N \lambda^{-N} \; \lambda^{n} 
\end{align*}
for all $(x,y) \in \cU$.  Since the later implies that 
\begin{align}\label{eq:upperbound}
\limsup_{n\to\infty} \frac1n \log  \left\| D\psi^n(x,y) \frac{\partial}{\partial y} \right\| 
	& \leq \log \la<0
\end{align}
this finishes the proof of first statement of the proposition.

Now, since $T$ has negative Schwarzian derivative one can use  a theorem by Singer to deduce that the orbit
of any critical point belongs to the topological basin of attraction of some periodic attractor. Since there are
finitely many periodic attractors for $T$ then there exists $K\ge 1$ such that
$T^K(c)\in U$ for every $c$ such that $T'(c)=0$. By continuity,  there exists an open neighbourhood  $V$ of 
the critical points for $T$ such that $\psi^K (X \times V) \subset \cU$, provided that $\al$ is small enough.
This finishes the proof of the proposition.
\end{proof}

\begin{remark}
Let us mention that the set $\bigcup_i \big\{   (x, h_x( T^j(p_i) )) : x\in X , 1\le j \le \pi(p_i) \big\}
\subset \cU$
of curves formed by the continuations of the periodic attractors is not necessarily $\psi$-invariant 
(e.g. if $a$ is a Morse function it causes implies a transversality condition for admissible curves). 
An interesting question is to determine if, under some conditions on the function $a$, the positively 
invariant set  $\mathcal G=\bigcap_{n\ge 0} \psi^n(\cU)$ support an absolutely continuous SRB measure.
\end{remark}

We now describe the asymptotics of Lebesgue almost every point that is mapped in 
the trapping region $\cU$. More precisely we prove the following:

\begin{proposition}
There exists $\vep>0$ such that if $\|a\|_{C^3}<\vep$ then there exists an SRB measure $\nu$ for $\psi$, it 
is  hyperbolic and its basin of attraction contains Lebesgue almost every point in some open subset of $\cU$.
\end{proposition}

\begin{proof}
Let $\mu_S$ denote the unique SRB measure for $S$ and $\vep$ and the trapping region $\cU=X\times U$ 
as in Proposition~\ref{prop:hyp.saddle}. Take $y_0\in U$ arbitrary and consider the sequence of probability
 measures given by
$$
\nu_n=\frac1n\sum_{j=0}^{n-1} \psi^j_* ( \mu_S \times \delta_{y_0}),
	\quad n\ge 1.
$$
By construction, $x_0\in U_i$ for some $i$ and thus belongs to the topological basin of atraction of $p_i$
with respect to the polynomial $T$.

We claim that $(\nu_n)_n$ is convergent to an $\varphi$-invariant, ergodic probability measure $\nu$ whose 
basin of attraction covers Lebesgue almost every point in the open set $X\times U_i$. 
On the one hand, if one considers the projection $\pi_X: X\times I \to X$ on the first coordinate then 
the sequence $(\pi_X)_*\nu_n=\frac1n \sum_{j=0}^{n-1} S^j_* \mu_S=\mu_S$ for all $n$.
On the other hand, the uniform contraction on $X\times U_i$ under iterations by $T_x^n$ implies that for any 
continuous observable $G$, $\delta>0$ and points $y_1, y_2 \in \{x\}\times U_i$ the Birkhoff averages satisfy
$|\frac1n\sum_{j=0}^{n-1}  G \circ \psi^j(x,y_1) - \frac1n\sum_{j=0}^{n-1}  G \circ \psi^j(x,y_2)|
	<\delta$
provided that $n$ is large (depending only on $G$ and $\lambda$). 
Therefore,  the functional $\Psi: C( X \times \overline{U_i})\to \mathbb R$ given by
\begin{align}\label{eq.measurecontract}
\Psi(G) & = \lim_{n\to\infty}  
	\frac1n\sum_{j=0}^{n-1} \int \max_{y\in \overline{U_i}} [G \circ \psi^j(x,y) ] \; d\mu_S(x)\\
	  & = \lim_{n\to\infty}  
	\frac1n\sum_{j=0}^{n-1} \int G \circ \psi^j(x,y_0)  \; d\mu_S(x) \\
	& = \lim_{n\to\infty}  \int G \; d\nu_n
\end{align}
is well defined and  clearly continuous. By Riesz representation theorem there exists a probability measure
$\nu$ with $\supp(\nu)\subset X\times \overline{U_i}$ and such that $\Psi(G)=\int G \,d\nu$ for all continuous function $G$.
Furthermore, $(\pi_X)_*\nu=\mu_X$ and, since $\mu_X$ is $S$-invariant and ergodic, the limit \eqref{eq.measurecontract} is almost everywhere constant in $x$ and consequently
$$
\int G \, d\nu
	= \lim_{n\to+\infty} \frac1n\sum_{j=0}^{n-1} G \circ \psi^j(x,y)  
$$
for $\mu_S\times \Leb$-almost every $(x,y)\in X \times \overline{U_i}$. In particular, this proves that $\nu$ is 
$\psi$-invariant and ergodic, and that Lebesgue almost every $(x,y)\in X\times \overline{U_i}$ belongs to 
the basin of attraction $B(\nu)$ and so $\nu$ is the unique SRB measure supported in $X\times \overline{U_i}$.
\end{proof}

Hence to finalize the proof of Theorem~\ref{thm:hiperbolicidade.fibrada} let us assume that 
$\mu_S$ has only non-zero Lyapunov exponents and prove that the SRB measure $\nu$ is
hyperbolic. This holds from the fact that the Lyapunov exponents for $\nu$ coincide with the ones
of $\mu_S$ together with the fibered Lyapunov exponent, which is negative since it is bounded from above by \eqref{eq:upperbound} since $\nu$ is supported in $X\times \overline{U_i}$. This
finishes the proof of the theorem.

\begin{remark}
One final remark is that it is not hard to check from the proof above that there are exactly as many hyperbolic SRB measures supported in $\cU$ as the number of periodic attracting points for the hyperbolic interval map $T$. This comes from the fact that the measure $\delta_{y_0}$ is associated to a point $y_0$ that belongs to the topological
basin of attraction of some of the periodic attractors. 
\end{remark}

\subsection*{Acknowledgments} The author is grateful to J. F. Alves, J. Rivera-Letelier, D. Schnellmann, S. Senti, O. Sester and M. Viana for their comments and inspiring conversations.  This work was partially supported by CNPq and FAPESB.

\bibliographystyle{alpha}

\end{document}